\documentclass[12pt,reqno]{article}
\usepackage[usenames]{color}
\usepackage[colorlinks=true,linkcolor=webgreen,filecolor=webbrown,citecolor=webgreen]{hyperref}
\usepackage{amssymb}
\usepackage{amsmath}
\usepackage{amsfonts}
\usepackage[T1]{fontenc}
\usepackage{enumerate}
\usepackage{graphicx}
\usepackage{color}

\definecolor{webgreen}{rgb}{0,.5,0}
\definecolor{webbrown}{rgb}{.6,0,0}
\newtheorem{theorem}{Theorem}

\newenvironment{proof}[1][Proof]{\noindent\textbf{#1.} }{\ \rule{0.5em}{0.5em}}

\allowdisplaybreaks

\begin{document}

\begin{center}
\vskip1cm

{\LARGE \textbf{A representation for the integral kernel of the composition
of multivariate Bernstein--Durrmeyer operators}}

\vspace{2cm}

{\large Ulrich Abel}\\[0pt]
\textit{Technische Hochschule Mittelhessen}\\[0pt]
\textit{Fachbereich MND}\\[0pt]
\textit{Wilhelm-Leuschner-Stra\ss e 13, 61169 Friedberg }\\[0pt]
\textit{Germany}\\[0pt]
\href{mailto:Ulrich.Abel@mnd.thm.de}{\texttt{Ulrich.Abel@mnd.thm.de}}

\vspace{1cm}

{\large Ana Maria Acu}\\[0pt]
\textit{Lucian Blaga University of Sibiu}\\[0pt]
\textit{Department of Mathematics and Informatics}\\[0pt]
\textit{Romania}\\[0pt]
\textit{Str. Dr. Ion Ra\c tiu, No. 5-7, 550012 Sibiu}\\[0pt]
\textit{Romania}\\[0pt]
\href{mailto:anamaria.acu@ulbsibiu.ro}{\texttt{anamaria.acu@ulbsibiu.ro}}

\vspace{1cm}

{\large Margareta Heilmann}\\[0pt]
\textit{University of Wuppertal}\\[0pt]
\textit{School of Mathematics and Natural Sciences}\\[0pt]
\textit{Gau\ss stra\ss e 20, 42119 Wuppertal}\\[0pt]
\textit{Germany}\\[0pt]
\href{mailto:heilmann@math.uni-wuppertal.de}{\texttt{%
heilmann@math.uni-wuppertal.de}}

\vspace{1cm}

{\large Ioan Ra\c sa}\\[0pt]
\textit{Technical University of Cluj-Napoca}\\[0pt]
\textit{Faculty of Automation and Computer Science, Department of Mathematics%
}\\[0pt]
\textit{Str. Memorandumului nr. 28, 400114 Cluj-Napoca}\\[0pt]
\textit{Romania}\\[0pt]
\href{mailto:ioan.rasa@math.utcluj.ro}{\texttt{ioan.rasa@math.utcluj.ro}}
\end{center}

\vspace{2cm}

\newpage

\begin{abstract}
This paper presents a representation for the kernel of the composition of
multivariate Bernstein--Durrmeyer operators for functions defined on the
standard simplex in $\mathbb{R}^{d}$.
\end{abstract}

\bigskip

\bigskip



\vspace{2cm}

\section{Introduction}

\label{intro}

In their recent paper \cite{AAHR-Kernels for composition-JMAA-2026} the
authors studied the composition of Bernstein--Durrmeyer operators $%
M_{n}:L_{1}[0,1]\longrightarrow \Pi _{n}$ defined by 
\begin{eqnarray*}
(M_{n}f)(x) &=&\int_{0}^{1}f(y)K_{n}(x,y)dy,\,\,x\in \lbrack 0,1], \\
K_{n}(x,y) &=&(n+1)\sum_{k=0}^{n}p_{n,k}(x)p_{n,k}(y),\,p_{n,k}(x)=\binom{n}{%
k}x^{k}(1-x)^{n-k},
\end{eqnarray*}%
focusing on the structure and properties of the associated kernel functions.
They established new identities for the kernel arising from the composition
of two and three operators. For the composition of $2$ classical
Bernstein--Durrmeyer operators one has the following formula \cite[Theorem~1]%
{AAHR-Kernels for composition-JMAA-2026}:\ 

\begin{theorem}
\label{theorem-beautyful-representation-r=2}For $m,n\geq 0$, the kernel of
the composition of classical Bernstein--Durrmeyer operators $M_{m}\circ
M_{n} $ has the representation 
\begin{equation*}
K_{m,n}\left( x,y\right) =\frac{\left( m+1\right) !\left( n+1\right) !}{%
\left( m+n+1\right) !}\sum_{k=0}^{\min \{m,n\}}\binom{m}{k}\binom{n}{k}%
\sum_{\ell =0}^{k}p_{k,\ell }\left( x\right) p_{k,\ell }\left( y\right) .
\end{equation*}
\end{theorem}

Like the well-known representation in terms of Legendre polynomials, which
can be written in the form 
\begin{eqnarray*}
K_{m,n}\left( x,y\right) &=&\sum_{k=0}^{\min \left\{ m,n\right\} }\frac{m^{%
\underline{k}}}{\left( m+k+1\right) ^{\underline{k}}}\frac{n^{\underline{k}}%
}{\left( n+k+1\right) ^{\underline{k}}}\left( 2k+1\right) \\
&&\times \sum_{i=0}^{k}\left( -1\right) ^{i}\binom{k}{i}p_{k,i}\left(
x\right) \sum_{j=0}^{k}\left( -1\right) ^{j}\binom{k}{j}p_{k,j}\left(
y\right)
\end{eqnarray*}%
(see \cite{AAHR-Kernels for composition-JMAA-2026}), Theorem~\ref%
{theorem-beautyful-representation-r=2} shows the commutativity of these
operators naturally. While the above Legendre representation contains all
possible products $p_{k,i}\left( x\right) p_{k,j}\left( y\right) $, $0\leq
i,j\leq k\leq n$, of Bernstein basis polynomials, the new representation has
the beautiful property to contain only products $p_{k,\ell }\left( x\right)
p_{k,\ell }\left( y\right) $, $0\leq \ell \leq k\leq n$, where $n$ is the
smallest degree of the Bernstein--Durrmeyer polynomials involved. This fact
immediately implies that the composition can be written as a linear
combination of the operators themselves.

For the composition of $3$ classical Bernstein--Durrmeyer operators the
authors showed the following formula \cite[Theorem~2]{AAHR-Kernels for
composition-JMAA-2026}:\ 

\begin{theorem}
\label{theorem-beautyful-representation-r=3}The kernel of the composition of 
$3$ Bernstein--Durrmeyer operators $M_{n_{3}}\circ M_{n_{2}}\circ M_{n_{1}}$
has the representation 
\begin{align*}
K_{n_{3},n_{2},n_{1}}\left( x,y\right) & =\frac{\left( n_{3}+1\right)
!\left( n_{2}+1\right) !\left( n_{1}+1\right) !\left(
n_{3}+n_{2}+n_{1}+1\right) !}{\left( n_{3}+n_{2}+1\right) !\left(
n_{3}+n_{1}+1\right) !\left( n_{2}+n_{1}+1\right) !} \\
& \times \sum_{k=0}^{\min \{n_{1},n_{2},n_{3}\}}\frac{\binom{n_{3}}{k}\binom{%
n_{2}}{k}\binom{n_{1}}{k}}{\binom{n_{3}+n_{2}+n_{1}+1}{k}}\sum_{\ell
=0}^{k}p_{k,\ell }\left( x\right) p_{k,\ell }\left( y\right) .
\end{align*}
\end{theorem}

In this paper we prove a generalization of Theorem~\ref%
{theorem-beautyful-representation-r=2} in the multivariate case. In a
forth-coming paper the corresponding generalization of Theorem~\ref%
{theorem-beautyful-representation-r=3} will be given. Moreover, the authors
will derive a representation of its $r$-th iterate of the multivariate
Bernstein--Durrmeyer operator $M_{n}$ as a linear combination of the
multivariate operators $M_{k}$, for $k=0,1,\ldots ,n$.

\section{Notation}

We start with notation. Let $d$ be a positive integer. The standard simplex
in $\mathbb{R}^{d}$ is given by 
\begin{equation*}
\mathbb{S}^{d}:=\left \{ x=\left( x_{1},\ldots ,x_{d}\right) \in \mathbb{R}%
^{d}:\ 0\leq x_{1},\ldots ,x_{d}\leq 1,\ x_{1}+\cdots +x_{d}\leq 1\right \} .
\end{equation*}%
We will also use the barycentric coordinates 
\begin{equation*}
\mathbf{x}=\left( x_{0},x_{1},\ldots ,x_{d}\right) ,\qquad
x_{0}:=1-x_{1}-\cdots -x_{d}.
\end{equation*}%
Throughout the paper, we will use standard multi-index notation. For
example, for $x=(x_{1},\ldots ,x_{d})\in \mathbb{R}^{d}$ and $\beta =(\beta
_{1},\ldots ,\beta _{d})\in \mathbb{Z}^{d}$ we write 
\begin{equation*}
x^{\beta }:=x_{1}^{\beta _{1}}\cdots x_{d}^{\beta _{d}},\qquad \left \vert
\beta \right \vert :=\beta _{1}+\cdots +\beta _{d},\qquad \beta !:=\beta
_{1}!\cdots \beta _{d}!.
\end{equation*}%
Multinomial coefficients are defined by ${\binom{\left \vert \beta
\right
\vert }{\beta }}:=\frac{|\beta |!}{\beta !}$ if $\beta =(\beta
_{1},\ldots ,\beta _{d})\in \mathbb{N}_{0}^{d}$ and ${\binom{\left \vert
\beta \right
\vert }{\beta }}:=0$ if (at least)\ one of the components $%
\beta _{\nu }$ is negative. We will also use binomial coefficients ${\binom{s%
}{m}}$ with $s\in \mathbb{R}$ and $m\in \mathbb{N}_{0}$ which are defined by 
${\binom{s}{m}}:=\frac{s\left( s-1\right) \cdots \left( s-m+1\right) }{m!}$
for $m\in \mathbb{N}$ and ${\binom{s}{0}}:=1$.

For given $\alpha =\left( \alpha _{0},\alpha _{1},\ldots ,\alpha _{d}\right)
\in \mathbb{N}_{0}^{d+1}$, the $d$-variate Bernstein basis polynomials are 
\begin{equation*}
B_{\alpha }\left( x\right) :={\binom{\left\vert \alpha \right\vert }{\alpha }%
}\mathbf{x}^{\alpha }=\frac{\left\vert \alpha \right\vert !}{\alpha
_{0}!\alpha _{1}!\cdots \alpha _{d}!}\,x_{0}^{\alpha _{0}}x_{1}^{\alpha
_{1}}\cdots x_{d}^{\alpha _{d}}.
\end{equation*}%
We use the inner product 
\begin{equation*}
\left\langle f,g\right\rangle :=\int_{\mathbb{S}^{d}}f\left( x\right)
\,g\left( x\right) \,dx.
\end{equation*}%
For a function $f$ on $\mathbb{S}^{d}$, the Bernstein-Durrmeyer operator of
degree $n\in \mathbb{N}$ is defined by the formula 
\begin{equation*}
\left( M_{n}f\right) \left( x\right) :=\sum_{\left\vert \alpha \right\vert
=n}\frac{\left\langle f,B_{\alpha }\right\rangle }{\left\langle \mathbf{1}%
,B_{\alpha }\right\rangle }B_{\alpha }\left( x\right) .
\end{equation*}%
Here, $\mathbf{1}$ denotes the constant function equal to one. The
Bernstein-Durrmeyer operator is usually considered on the domain $%
L^{p}\left( \mathbb{S}^{d}\right) $, $1\leq p<\infty $, which is the
weighted $L^{p}$ space consisting of all measurable functions on $\mathbb{S}%
^{d}$ such that 
\begin{equation*}
\Vert f\Vert _{p}:=\left( \int_{\mathbb{S}^{d}}\left\vert f\left( x\right)
\right\vert ^{p}\,dx\right) ^{1/p}
\end{equation*}%
is finite. For $p=\infty $, the space $C\left( \mathbb{S}^{d}\right) $ of
continuous functions is considered instead. Note that the spaces $%
L^{p}\left( \mathbb{S}^{d}\right) $, $1<p<\infty $, and $C\left( \mathbb{S}%
^{d}\right) $ are subspaces of $L^{1}\left( \mathbb{S}^{d}\right) $.

\section{The main result and its proof}

For $x,y\in \mathbb{S}^{d}$, we define the multivariate kernel $%
K_{m,n}\left( x,y\right) $ via the definition of $M_{m}\circ M_{n}$: 
\begin{eqnarray*}
&&\left( \left( M_{m}\circ M_{n}\right) f\right) \left( x\right)  \\
&=&\int_{\mathbb{S}^{d}}f\left( y\right) \sum_{\left\vert \beta \right\vert
=m}\sum_{\left\vert \alpha \right\vert =n}B_{\alpha }\left( y\right)
B_{\beta }\left( x\right) \frac{1}{\left\langle \mathbf{1},B_{\alpha
}\right\rangle \left\langle \mathbf{1},B_{\beta }\right\rangle }\left( \int_{%
\mathbb{S}^{d}}B_{\alpha }\left( t\right) B_{\beta }\left( t\right)
dt\right) dy \\
=: &&\int_{\mathbb{S}^{d}}f\left( y\right) K_{m,n}\left( x,y\right) dy,
\end{eqnarray*}%
where $\mathbf{1=}\left( 1,1,\ldots ,1\right) \in \mathbb{N}^{d+1}$. The
following theorem is the main result of this note.

\begin{theorem}
For positive integers $m,n$ and $x,y\in \mathbb{R}^{d}$, the kernel $%
K_{m,n}\left( x,y\right) $ has the representation 
\begin{equation*}
K_{m,n}\left( x,y\right) =\frac{\left( m+d\right) !\left( n+d\right) !}{%
\left( m+n+d\right) !}\sum_{\ell \geq 0}\binom{m}{\left \vert \ell \right
\vert }\binom{n}{\left \vert \ell \right \vert }B_{\ell }\left( \mathbf{x}%
\right) B_{\ell }\left( \mathbf{y}\right) .
\end{equation*}
\end{theorem}

\begin{proof}
For positive integers $m,n$ and $x,y\in \mathbb{R}^{d}$, we have 
\begin{equation*}
K_{m,n}\left( x,y\right) =\sum_{\left\vert \beta \right\vert
=m}\sum_{\left\vert \alpha \right\vert =n}B_{\alpha }\left( y\right)
B_{\beta }\left( x\right) \frac{1}{\left\langle \mathbf{1},B_{\alpha
}\right\rangle \left\langle \mathbf{1},B_{\beta }\right\rangle }\int_{%
\mathbb{S}^{d}}B_{\alpha }\left( t\right) B_{\beta }\left( t\right) dt.
\end{equation*}%
Note that $\left\langle \mathbf{1},B_{\alpha }\right\rangle =\left\vert
\alpha \right\vert !/\left( \left\vert \alpha \right\vert +d\right) !$ (see,
e.g., \cite[Page~147]{Berens-Schmidt-Xu-JAT-1992}). Since 
\begin{equation*}
\int_{S^{d}}B_{\alpha }\left( t\right) B_{\beta }\left( t\right) dt=\frac{%
\binom{\left\vert \alpha \right\vert }{\alpha }\binom{\left\vert \beta
\right\vert }{\beta }}{\binom{\left\vert \alpha +\beta \right\vert }{\alpha
+\beta }}\left\langle \mathbf{1},B_{\alpha +\beta }\right\rangle 
\end{equation*}%
we obtain 
\begin{eqnarray*}
K_{m,n}\left( x,y\right)  &=&\frac{\left( m+d\right) !\left( n+d\right)
!\left( m+n\right) !}{m!n!\left( m+n+d\right) !}\sum_{\left\vert \beta
\right\vert =m}\sum_{\left\vert \alpha \right\vert =n}B_{\alpha }\left(
y\right) B_{\beta }\left( x\right) \frac{\binom{\left\vert \alpha
\right\vert }{\alpha }\binom{\left\vert \beta \right\vert }{\beta }}{\binom{%
\left\vert \alpha +\beta \right\vert }{\alpha +\beta }} \\
&=&\frac{\left( m+d\right) !\left( n+d\right) !}{\left( m+n+d\right) !}%
\sum_{\left\vert \beta \right\vert =m}\frac{1}{\beta !}B_{\beta }\left(
x\right) \sum_{\left\vert \alpha \right\vert =n}B_{\alpha }\left( y\right) 
\frac{\left( \alpha +\beta \right) !}{\alpha !}.
\end{eqnarray*}%
With the notation\ $\mathbf{t}=\left( t_{0},t_{1},\ldots ,t_{d}\right) $\
and $\mathbf{D}^{\beta }=\prod_{\nu =0}^{d}\left( \frac{\partial }{\partial
t_{\nu }}\right) ^{\beta _{\nu }}$, the inner sum is equal to 
\begin{eqnarray*}
\sum_{\left\vert \alpha \right\vert =n}B_{\alpha }\left( y\right) \frac{%
\left( \alpha +\beta \right) !}{\alpha !} &=&\sum_{\left\vert \alpha
\right\vert =n}B_{\alpha }\left( y\right) \left( \mathbf{D}^{\beta }\mathbf{t%
}^{\alpha +\beta }\right) _{\mid \mathbf{t=1}}=\mathbf{D}^{\beta }\left( 
\mathbf{t}^{\beta }\sum_{\left\vert \alpha \right\vert =n}B_{\alpha }\left(
y\right) \mathbf{t}^{\alpha }\right) _{\mid \mathbf{t=1}} \\
&=&\left( \prod_{\nu =0}^{d}\left( \frac{\partial }{\partial t_{\nu }}%
\right) ^{\beta _{\nu }}\left( t_{0}^{\beta _{0}}t_{1}^{\beta _{1}}\cdots
t_{d}^{\beta _{d}}\left( y_{0}t_{0}+y_{1}t_{1}+\cdots +y_{d}t_{d}\right)
^{n}\right) \right) _{\mid \mathbf{t=1}} \\
&=&\sum_{\ell \leq \beta }n^{\underline{\left\vert \ell \right\vert }%
}\prod_{\nu =0}^{d}\left( \binom{\beta _{\nu }}{\ell _{\nu }}\frac{\beta
_{\nu }!}{\ell _{\nu }!}y_{\nu }^{\ell _{\nu }}\right)  \\
&=&\sum_{\ell \leq \beta }n^{\underline{\left\vert \ell \right\vert }}\frac{1%
}{\left\vert \ell \right\vert !}B_{\ell }\left( y\right) \beta !\prod_{\nu
=0}^{d}\binom{\beta _{\nu }}{\ell _{\nu }}.
\end{eqnarray*}%
Here, $\ell \leq \beta $ means summation over all nonnegative integers $\ell
_{\nu }$ such that $\ell _{\nu }\leq \beta _{\nu }$, for $\nu =0,\ldots ,d$.
Therefore, 
\begin{eqnarray*}
&&\left( \frac{\left( m+d\right) !\left( n+d\right) !}{\left( m+n+d\right) !}%
\right) ^{-1}K_{m,n}\left( x,y\right)  \\
&=&\sum_{\left\vert \beta \right\vert =m}B_{\beta }\left( x\right)
\sum_{\ell \leq \beta }\binom{n}{\left\vert \ell \right\vert }B_{\ell
}\left( y\right) \prod_{\nu =0}^{d}\binom{\beta _{\nu }}{\ell _{\nu }} \\
&=&\sum_{\ell \geq 0}\binom{n}{\left\vert \ell \right\vert }B_{\ell }\left(
y\right) \sum_{\left\vert \beta \right\vert =m}B_{\beta }\left( x\right)
\prod_{\nu =0}^{d}\binom{\beta _{\nu }}{\ell _{\nu }} \\
&=&\sum_{\ell \geq 0}\binom{n}{\left\vert \ell \right\vert }B_{\ell }\left(
y\right) \sum_{\left\vert \beta \right\vert =m}{\binom{\left\vert \beta
\right\vert }{\beta }}\mathbf{x}^{\beta }\prod_{\nu =0}^{d}\frac{\left(
\beta _{\nu }\right) ^{\underline{\ell _{\nu }}}}{\ell _{\nu }!} \\
&=&\sum_{\ell \geq 0}\binom{n}{\left\vert \ell \right\vert }B_{\ell }\left(
y\right) m^{\underline{\left\vert \ell \right\vert }}\left( \prod_{\nu
=0}^{d}\frac{1}{\ell _{\nu }!}\right) \sum_{\substack{ \left\vert \beta
\right\vert =m \\ \beta \geq \ell }}{\binom{m-\left\vert \ell \right\vert }{%
\beta -\ell }}\mathbf{x}^{\beta },
\end{eqnarray*}%
where $\ell \geq 0$ means summation over all nonnegative integers $\ell
_{\nu }$ $\left( \nu =0,\ldots ,d\right) $. Observing that 
\begin{equation*}
\sum_{\substack{ \left\vert \beta \right\vert =m \\ \beta \geq \ell }}{%
\binom{m-\left\vert \ell \right\vert }{\beta -\ell }}\mathbf{x}^{\beta
}=\sum_{\left\vert \beta \right\vert =m-\left\vert \ell \right\vert }{\binom{%
m-\left\vert \ell \right\vert }{\beta }}\mathbf{x}^{\beta +\ell }=\mathbf{x}%
^{\ell }\sum_{\left\vert \beta \right\vert =m-\left\vert \ell \right\vert }{%
\binom{m-\left\vert \ell \right\vert }{\beta }}B_{\beta }\left( x\right) =%
\mathbf{x}^{\ell }
\end{equation*}%
and using that $\prod_{\nu =0}^{d}\frac{1}{\ell _{\nu }!}=\frac{1}{%
\left\vert \ell \right\vert !}\binom{\left\vert \ell \right\vert }{\ell }$
we obtain 
\begin{equation*}
\left( \frac{\left( m+d\right) !\left( n+d\right) !}{\left( m+n+d\right) !}%
\right) ^{-1}K_{m,n}\left( x,y\right) =\sum_{\ell \geq 0}\binom{m}{%
\left\vert \ell \right\vert }\binom{n}{\left\vert \ell \right\vert }B_{\ell
}\left( y\right) \binom{\left\vert \ell \right\vert }{\ell }\mathbf{x}^{\ell
}.
\end{equation*}%
Since $\binom{\left\vert \ell \right\vert }{\ell }\mathbf{x}^{\ell }=B_{\ell
}\left( x\right) $ the proof is completed.
\end{proof}

\end{document}